\documentclass[12pt,letterpaper,twoside]{article}

\usepackage{amsmath,amssymb}
\usepackage{amsthm}
\usepackage{mathtools}
\usepackage{stmaryrd}
\usepackage{fancyhdr}
\usepackage{times}
\usepackage{mathrsfs} % Required for C and H
\usepackage{titlesec}
\usepackage[normalem]{ulem}
\usepackage[pdftex,dvips]{geometry}
\mathtoolsset{centercolon,mathic}

\newtheoremstyle{fact}% name
     {\topsep}%      Space above
     {\topsep}%      Space below
     {\slshape}%         Body font
     {}%         Indent amount (empty = no indent, \parindent = para indent)
     {\bfseries}% Thm head font 
     {}%        Punctuation after thm head
     { }%     Space after thm head: " " = normal interword space;
     {\thmname{#1}\thmnumber{ #2.}\thmnote{ \rm (#3)}}
\newtheoremstyle{mylabel}% name
     {\topsep}%      Space above
     {\topsep}%      Space below
     {\itshape}%         Body font
     {}%         Indent amount (empty = no indent, \parindent = para indent)
     {\bfseries}% Thm head font 
     {}%        Punctuation after thm head
     { }%     Space after thm head: " " = normal interword space;
     {\thmname{#1}\thmnote{ #3}.}
\newtheorem{theorem}{Theorem}[section]
\newtheorem{Ltheorem}{Theorem}

\newtheorem*{theorem*}{Theorem} 
\newtheorem{lemma}[theorem]{Lemma}

\newtheorem{problem}[theorem]{Problem}

\newtheorem*{conjecture*}{Conjecture}

\theoremstyle{mylabel}
\newtheorem*{Ltheorem*}{Theorem}

\theoremstyle{definition}

\newtheorem{remark}[theorem]{Remark}

\newtheorem*{remark*}{Remark}

\newtheorem*{question*}{Question}
\newtheorem*{examples*}{Examples}  
\newtheorem{example}[theorem]{Example}

\newtheorem*{example*}{Example}

\newtheorem*{convention*}{Convention}

\theoremstyle{fact}

\newtheorem{ftheorem}[theorem]{Theorem}
\newtheorem{flemma}[theorem]{Lemma}

\newenvironment{myromanlist}[1][enumi]{\begin{list}{{\rm (\roman{#1})}}
{\usecounter{#1}\setlength{\labelwidth}{25pt}\setlength{\topsep}{-6pt}
\setlength{\itemsep}{-4pt} \setlength{\leftmargin}{25pt}}}{\end{list}}

\newenvironment{myalphlist}[1][enumi]{\begin{list}{{\rm (\alph{#1})}}
{\usecounter{#1}\setlength{\labelwidth}{25pt}\setlength{\topsep}{-6pt}
\setlength{\itemsep}{-4pt} \setlength{\leftmargin}{25pt}}}{\end{list}}

\def\proofont{\fontseries{bx}\fontshape{sc}\selectfont}
\def\proofname{Proof.}

\newcommand{\Note}[1]{}

\makeatletter
\renewenvironment{proof}[1][\proofname]{\par
  \normalfont
  \topsep6\p@\@plus6\p@ \trivlist
  \item[\hskip\labelsep\noindent\proofont #1]\ignorespaces
}{%
  \qed\endtrivlist
}
\makeatother

\titlelabel{\thetitle.\ }
\titleformat*{\section}{\normalsize\bfseries\centering}
\titleformat*{\subsection}{\normalsize\bfseries\itshape}
\author{D. Dikranjan\thanks{The first author acknowledges the financial 
aid received from grant MTM2009-14409-C02-01.} { }and 
G\'abor Luk\'acs\thanks{The second author gratefully acknowledges the 
generous financial support received from NSERC, which enabled him to do 
this research.}}

\title{On the quasi-component of\\ 
pseudocompact abelian groups
}

\begin{document}

\makeatletter
\def\@fnsymbol#1{\ifcase#1\or * \or 1 \or 2  \else\@ctrerr\fi\relax}
\let\mytitle\@title

\chead{\small\itshape D. Dikranjan and G. Luk\'acs / 
On the quasi-component of minimal pseudocompact abelian groups}
\fancyhead[RO,LE]{\small \thepage}
\makeatother

\maketitle

\def\thanks#1{} 

\thispagestyle{fancy}

\begin{abstract} 
In this paper, we describe the relationship between the quasi-component 
$q(G)$ of a (perfectly) minimal pseudocompact abelian group $G$ and the 
quasi-component $q(\widetilde G)$ of its completion. Specifically, we 
characterize the  pairs $(C,A)$ of compact connected abelian groups $C$ 
and subgroups $A$ such that $A \cong  q(G)$ and $C \cong  q(\widetilde G)$. 
As a consequence,we show that for every positive integer $n$  or
$n=\omega$, there exist plenty of  abelian pseudocompact perfectly minimal 
$n\mbox{-}$dimensional groups $G$ such that the quasi-component of~$G$ is 
not dense in the quasi-component of the completion of $G$.
\end{abstract}

\section{Introduction}

\label{sect:intro}

\thispagestyle{empty}

A Tychonoff space $X$ is {\itshape pseudocompact} if every continuous 
real-valued map on $X$ is bounded. If a~topological group
$G$ is pseudocompact, then its 
completion $\widetilde G$ is compact, that is, $G$ is {\itshape 
precompact} (cf.~\cite[1.1]{ComfRoss2}), which allows for the following 
characterization of pseudocompact groups.

\begin{ftheorem}[\cite{ComfRoss2}] 
\label{prel:thm:lcps}
A topological group $G$ is pseudocompact if and 
only if $G$ is precompact and $G_\delta$-dense in $\widetilde G$,
in which case 
$\widetilde G  =  \beta G$.
\end{ftheorem}

A Tychonoff space is {\itshape zero-dimensional} if it has a base 
consisting of {\itshape clopen} (open-and-closed)  sets.
For a topological group $G$, let $G_0$ and $q(G)$ denote the connected 
component and the quasi-component of the identity, respectively. For every 
group $G$, the quotient $G/G_0$ is {\itshape hereditarily disconnected} 
(i.e., $(G/G_0)_0$ is trivial),  and $G/q(G)$ is {\itshape totally 
disconnected} (i.e., $q(G/q(G))$ is trivial). While in the class of 
locally compact groups (or spaces) these notions of disconnectedness 
coincide with zero-dimensionality, this is  not the case in general. 
Indeed, both $G/G_0$ and $G/q(G)$ may fail to be zero-dimensional, or  to 
admit a coarser zero-dimensional group topology, even in the presence of 
additional compactness-like properties (cf. \cite[7.7]{ComfvMill2}, 
\cite{Megrel4}, \cite[Theorem~D]{DikGL3}). There is a  close relationship 
between connectedness and disconnectedness properties of pseudocompact 
groups and those of their completions, which is summarized in the next 
theorem.

\begin{ftheorem} \label{prel:thm:connsum}
Let $G$ be a pseudocompact group. Then:

\begin{myalphlist}

\item
{\rm (\cite[1.4]{Dikconpsc})}
$q(G)= q(\widetilde G) \cap G = (\widetilde G)_0 \cap  G$;

\item
{\rm (\cite{TkaDimLP})}
$G$ is zero-dimensional if and only if $\widetilde G$ is zero-dimensional;

\item
{\rm (\cite[1.7]{DikPS0dim})}
$G/q(G)$ is zero-dimensional if and only if $q(G)$ is dense in 
$(\widetilde G)_0=q(\widetilde G)$;

\item
{\rm (\cite[4.11(b)]{ComfGL})}
$G/G_0$ is zero-dimensional if and only if $G_0$ is dense in 
$(\widetilde G)_0=q(\widetilde G)$, in which case
$G_0 = q(G)$.

\end{myalphlist}
\end{ftheorem}

We say that a pair $(C,A)$ of a group $C$ and its subgroup $A$ is 
{\itshape realized} by a group $G$ if $A \cong q(G)$ and $C \cong 
q(\widetilde G)$. In a previous paper, the authors obtained sufficient 
conditions for a pair to be realizable by pseudocompact abelian groups 
that satisfy various degrees of so-called minimality properties 
(cf.~\cite[Theorem D$'$]{DikGL3}). In this paper, we obtain a complete 
characterization (i.e., necessary and sufficient conditions) for the same.

Recall that a (Hausdorff) topological group $G$ is {\itshape minimal} if 
there is no coarser (Hausdorff) group topology on $G$ (cf.~\cite{Steph} 
and~\cite{Doi}). For a pseudocompact group $G$, the quotient $G/q(G)$ is 
totally disconnected and pseudocompact, and thus, by an unpublished result 
of Shakhmatov, admits a coarser zero-dimensional group topology 
(cf.~\cite[Theorem~B]{DikGL3}). Therefore, Shakhmatov's result implies 
that for a pseudocompact group $G$, if $G/q(G)$ is minimal, then it is 
zero-dimensional. If in addition $G$ is minimal and abelian, then the 
converse is also true.

\begin{ftheorem}[{\cite[Theorem~3]{Dikdimpsc}, \cite[1.7]{DikPS0dim}}] 
\label{prel:thm:qG}
Let $G$ be a minimal pseudocompact abelian group. Then the following 
statements are equivalent:

\begin{myromanlist}

\item
$G/q(G)$ is zero-dimensional;

\item
$q(G)$ is dense in $(\widetilde G)_0=q(\widetilde G)$;

\item
$G/q(G)$ is minimal.

\end{myromanlist}
\end{ftheorem}

A group $G$ is {\itshape totally minimal} if every (Hausdorff) quotient of 
$G$ is totally minimal (cf.~\cite{DikPro}); $G$~is {\itshape perfectly} 
({\itshape totally}) {\itshape minimal} if the product $G \times H$ is 
({\itshape totally}) {\itshape minimal} for every ({\itshape totally}) 
{\itshape minimal} group $H$ (cf.~\cite{Stoy}). In an earlier paper, the 
authors proved the following:

\begin{ftheorem}[{\cite[Theorem~D$'$]{DikGL3}}] \label{prel:thm:D'}
Let $C$ be a connected compact abelian group, and $A$ a subgroup of~$C$.
Then there exists a pseudocompact abelian group $G$ such that 
$A \cong q(G)$ and $C \cong (\widetilde G)_0 = q(\widetilde G)$, and 
in particular, $\dim G =  \dim C$. Furthermore, if $A$ is dense in $C$ and

\begin{myalphlist}

\item
$A$ is minimal, then $G$ may be chosen to be  minimal;

\item
$A$ is totally minimal, then $G$ may be chosen to be totally 
minimal;

\item
$A$ is perfectly minimal, then $G$ may be chosen to be 
perfectly minimal;

\item
$A$ is perfectly totally minimal, then $G$ may be chosen to 
be perfectly totally minimal.

\end{myalphlist}
\end{ftheorem}

%%
%% An alternative way to state the theorem a bit more briefly
%%

%\begin{ftheorem}[{\cite[Theorem D$'$]{DikGL3}}]
%Let $C$ be a compact connected abelian group, and $A$ a \uline{dense} 
%subgroup of $C$. If $A$ is minimal (totally minimal, perfectly minimal, 
%perfectly totally minimal), then $(C,A)$ can be realized by a minimal 
%(totally minimal, perfectly minimal, perfectly totally minimal) 
%pseudocompact abelian group.
%\end{ftheorem}

By Theorem~\ref{prel:thm:qG}, if $G$ is a totally minimal pseudocompact 
abelian group, then $q(G)$ is dense in $(\widetilde G)_0 = 
q(\widetilde{G})$. Thus, the condition that $A$ is dense in $C$ is not 
only sufficient, but also necessary in parts (b) and (d) of 
Theorem~\ref{prel:thm:D'}. 

In this paper, we show that in order to realize $(C,A)$ by a~(perfectly) 
minimal pseudocompact abelian group, $A$~need not be dense in $C$, but a 
milder condition is both sufficient and necessary. A subgroup $E$ of a 
topological group $G$ is {\itshape essential}, and we put $E \leq_e G$, if 
for every non-trivial closed normal subgroup $N$ of $G$, the intersection 
$E \cap N$ is non-trivial.

\begin{Ltheorem}  \label{main:thm:realize}
Let $C$ be a compact connected abelian group, and $A$ a subgroup. Then:

\begin{myalphlist}

\item
$(C,A)$ can be realized by a minimal pseudocompact abelian group if and 
only if $A   \leq_e   C$;

\item
$(C,A)$ can be realized by a perfectly minimal pseudocompact abelian group 
if and only if $A$ is perfectly minimal and  
$A   \leq_e   C$;

\item
$(C,A)$ can be realized by a (perfectly) totally minimal pseudocompact 
abelian group  if and only if $A$ is (perfectly) totally minimal and dense 
in $C$.

\end{myalphlist}
\end{Ltheorem}

\begin{remark}
If $A$ is an essential subgroup of a compact group $C$, then in 
particular, it is essential in its completion $\widetilde A \subseteq C$, 
and thus $A$ is minimal (see Lemma~\ref{ess:lemma:ab-cl}(a) and 
Theorem~\ref{thm:min-crit}). Therefore, the condition $A \leq_e C$ in 
Theorem~\ref{main:thm:realize}(a) implies that $A$ is minimal.
\end{remark}

By applying Theorem~\ref{main:thm:realize} to the character group 
$\hom(\mathbb{Q},\mathbb{R}/\mathbb{Z})$ equipped with the topology of 
pointwise convergence and the annihilator of $\mathbb{Z}$ in it, we 
also obtain a family of ``pathological" examples.

\begin{Ltheorem} \label{main:thm:const}
For every positive integer $n$ or 
$n=\omega$,
there exists  an abelian pseudocompact perfectly minimal group $G_n$ with 
$\dim G_n  =  n$ such that $q(G_n)$ is 
not dense in $q(\widetilde G_n)=(\widetilde G_n)_0$, or 
equivalently, $G_n/q(G_n)$ is not minimal.
\end{Ltheorem}

The proofs of Theorems~\ref{main:thm:realize} and~\ref{main:thm:const} are 
presented in \S\ref{sect:proofs}; they are based on preservation 
properties of essential subgroups established in \S\ref{sect:ess}.

\section{Preliminaries: Essential and minimal subgroups}

\label{sect:ess}

While preservation of minimality under formation of closed subgroups and 
products has been thoroughly studied (cf.~\cite[Proposition~2.3, 
Lemma~3.1]{Prod2}, \cite[Th\'eor\`eme 1-2]{Doi}, \cite[9]{Steph2}, and 
\cite[(6), (3)]{EDS}), it appears that preservation of essentiality has 
not been well investigated in the context of topological groups. Our aim 
in this section is to remedy this state of affairs in the realm 
of abelian groups.

The relationship between minimality and essential subgroups was 
discovered independently by Stephenson and Prodanov, and generalized by 
Banaschewski (cf.~\cite[Theorem~2]{Steph}, \cite{Prod1}, and 
\cite[Propositions~1~and~2]{Banasch}).

\begin{ftheorem}[{\cite[2.5.1]{DikProSto}}, {\cite[3.21]{GLCLTG}}]
\label{thm:min-crit}
Let $G$ be a topological group, and $D$ a dense subgroup. Then $D$ is 
minimal if and only if $G$ is minimal and 
$D \leq_e  G$.
\end{ftheorem}

The celebrated Prodanov-Stoyanov Theorem states that every minimal abelian 
group is precompact (cf.~\cite{ProdStoj} and~\cite{ProdStoj2}), and  
allows for a complete characterization of minimality of abelian groups 
using the notion of an essential subgroup.

\begin{ftheorem}[{\cite[2.5.2]{DikProSto}}, {\cite[3.31]{GLCLTG}}]
\label{thm:min-ab}
An abelian topological group $G$ is minimal if and only if its completion 
$\widetilde G$ is compact and
$G \leq_e  \widetilde G$.
\end{ftheorem}

The next two easy lemmas, whose proofs have been omitted, describe
elementary properties of the essentiality relation.

\begin{lemma} \label{ess:lemma:gen-dense}
Let $G$ be a topological group, and $E$ and $H$ subgroups of $G$ such that
$E \subseteq H$.

\begin{myalphlist}

\item
If $E   \leq_e   H$ and
$H   \leq_e   G$, then 
$E   \leq_e   G$.

\item
If $H$ is dense in $G$, then
$E   \leq_e   G$ if and only if
$E   \leq_e   H$ and  
$H   \leq_e   G$.    \qed

\end{myalphlist}
\end{lemma}

\begin{lemma} \label{ess:lemma:ab-cl}
Let $G$ be an abelian topological group, and $E$ and $H$ subgroups. 

\begin{myalphlist}

\item
If $E   \subseteq   H$, then
$E   \leq_e   G$ if and only if
$E   \leq_e   H$ and
$H   \leq_e   G$.

\item
If $H$ is closed in $G$ and 
$E \leq_e G$, then
$E   \cap   H    
\leq_e   H$. \qed

\end{myalphlist}
\end{lemma}

\begin{remark} \label{ess:rem:ab-sub}
Every closed subgroup of a minimal abelian group is minimal
(cf.~\cite[Proposition~2.3]{Prod2}). Indeed,
let $G$ be a~minimal abelian group, and $M$ a closed subgroup. Then,
by Theorem~\ref{thm:min-ab}, $\widetilde G$ is compact and
$G   \leq_e   \widetilde G$.
So, by  Lemma~\ref{ess:lemma:ab-cl}(b),
$M  =   G   \cap   
(\operatorname{cl}_{\widetilde G} M) \leq_e
 \operatorname{cl}_{\widetilde G} M   =
  \widetilde M $. Thus, by 
Theorem~\ref{thm:min-ab}, $M$ is minimal. (In fact, every closed central 
subgroup of a minimal group is minimal; cf.~\cite[7.2.5]{DikProSto} 
and~\cite[3.26]{GLCLTG}.) It follows that every closed subgroup of a 
totally minimal abelian group is totally minimal, and in fact, every 
closed central subgroup of a totally minimal group is totally minimal 
(cf.~\cite[Lemma~3.1]{Prod2}, 
\cite[7.2.5]{DikProSto}, and \cite[3.27]{GLCLTG}).
 \end{remark}

We turn now to preservation of essentiality under formation of 
products.
A topological group $G$ is  {\itshape elementwise compact} if 
every $g  \in   G$ is contained in a 
compact subgroup of $G$, or equivalently, 
$\operatorname{cl}_G \langle g \rangle$ is compact for 
every $g  \in   G$ 
(cf.~\cite[5.4]{HofMor2}). Stephenson showed that every elementwise 
compact minimal group is perfectly minimal (cf.~\cite[9]{Steph2}). The 
next lemma is a natural extension of Stephenson's result.

\begin{lemma} \label{ess:lemma:ewcomp-prod}
Let $G$ be a topological group, and
$E  \leq_e   G$. 
If $L$ is an elementwise compact group, then
$L \times E \leq_e  L \times G$.
\end{lemma}

\begin{proof}
Let $\pi_2 \colon 
L \times G 
\rightarrow G$ denote the canonical 
projection, and let  $N$ be a non-trivial closed normal subgroup of
$L \times G$. 
If $\pi_2(N)$ is trivial,
then $N  \subseteq   
L \times   \{e\}$, and consequently
$N  \cap  
(L\times E)$ is non-trivial.
Thus, we may assume that $\pi_2(N)$ is non-trivial.
We prove the statement in two steps.

{\itshape Step 1.} Suppose that $L$ is compact.
Then $\pi_2$ is a closed map (cf.~\cite[3.1.16]{Engel6}), and so
$\pi_2(N)$ is a~closed normal subgroup of $G$. Thus,
$\pi_2(N) \cap  E$ is non-trivial, because 
$E  \leq_e   G$. Therefore, 
the intersection $N  \cap  
(L \times E)$ is non-trivial, as 
required.

{\itshape Step 2.} In the general case, let 
$x =  (l,g)  \in 
  N$ 
be such that 
$g  \neq  e$. Since $L$ is elementwise 
compact, there is a compact subgroup $S$ of $L$ such that
$l   \in  S$. Put
$N^\prime :=   N 
 \cap   
(S \times  G)$. Then $N^\prime$ is a closed 
normal subgroup of 
$S\times  G$,
and it is non-trivial, because 
$x \in N^\prime$. By what we 
have shown so far,
$N^\prime \cap  
(S \times  E)$ is non-trivial, and in 
particular, 
$N  \cap  
(L \times  E)$ is non-trivial, as 
desired.
\end{proof}

The next example shows that elementwise compactness cannot be replaced 
with precompactness, minimality, or completeness in 
Lemma~\ref{ess:lemma:ewcomp-prod}.

\begin{example}
Let $p$ be a prime, $\mathbb{Z}_p$ the group of $p$-adic integers, and 
$(\mathbb{Z},\tau_p)$ the integers equipped with the $p$-adic topology.
Since $(\mathbb{Z},\tau_p)$ is a minimal group whose completion is 
$\mathbb{Z}_p$ (cf.~\cite[2.5.6]{DikProSto}), by Theorem~\ref{thm:min-ab}, 
$\mathbb{Z}  \leq_e   \mathbb{Z}_p$. 
However,
$\mathbb{Z} \times \mathbb{Z}$ is not 
essential in 
$(\mathbb{Z},\tau_p) \times \mathbb{Z}_p$
or $\mathbb{Z} \times \mathbb{Z}_p$ 
(in the latter,  the first component is equipped with the discrete 
topology).
Indeed,  if $\xi   \in   
\mathbb{Z}_p\backslash\mathbb{Z}$, then 
$F :=  \langle (1,\xi)\rangle$ is a 
non-trivial closed subgroup of 
$(\mathbb{Z},\tau_p) \times \mathbb{Z}_p$
(and thus of 
$\mathbb{Z} \times \mathbb{Z}_p$)
such that 
$F \cap  
(\mathbb{Z} \times \mathbb{Z})$ is trivial.
\end{example}

It is also worth noting that preservation of essentiality does not imply 
precompactness, minimality, nor completeness, as the next example 
demonstrates.

\begin{example}
Let $p$ be a prime. Let $L_1$ denote the direct sum
$\mathbb{Z}_p^{(\omega)}$ equipped with the subgroup topology induced by
the direct product $\mathbb{Z}_p^{\omega}$, let
$L_2$ denote the direct sum $(\mathbb{R}/\mathbb{Z})^{(\omega)}$ 
equipped with the box topology, and put 
$L  :=  
L_1  \times  L_2$.
Then $L$ is elementwise compact, because elementwise compactness is 
preserved under formation of products, coproducts, sums, and 
$\Sigma$-products. Thus, by 
Lemma~\ref{ess:lemma:ewcomp-prod},  
$L$ {\itshape preserves essentiality}, that is, 
$E  \leq_e   G$ implies
$L \times E
 \leq_e  
L \times G$.
Nevertheless, $L$ is neither complete nor precompact, and in 
particular, by Theorem~\ref{thm:min-ab}, $L$ is not minimal. 
\end{example}

These examples provide a natural motivation for the following problem.

\begin{problem} \label{ess:prob:char}
Characterize the topological groups $L$ with the property that 
\begin{myalphlist}

\item
for every topological group $G$, if
$E  \leq_e   G$, then
$L \times E
 \leq_e 
L \times G$;

\item
for every abelian topological group $G$, if
$E  \leq_e   G$, then
$L \times E
 \leq_e 
L \times G$;

\item
for every topological group $G$, if  
$E  \leq_e   G$ and $E$ is 
minimal, then
$L \times E
 \leq_e 
L \times G$.

\end{myalphlist}
\end{problem}

We provide an answer to Problem~\ref{ess:prob:char}(c) in the special 
case where $L$ is minimal and abelian.

\begin{lemma} \label{ess:lemma:prodMin}
Let $M$ be a minimal abelian group, $G$ a topological group, and
$E  \leq_e   G$. If
$M  \times   E$ is minimal, then:

\begin{myalphlist}

\item
$M \times E \leq_e \widetilde M \times E$;

\item
$M \times E \leq_e \widetilde M \times G$;

\item
$M \times E \leq_e M \times G$.

\end{myalphlist}
\end{lemma}

\begin{proof}
(a) By Theorem~\ref{thm:min-crit}, 
$M  \times   E 
 \leq_e 
\widetilde M   \times   \widetilde E$, because
$M  \times   E$
is minimal. Consequently, by Lemma~\ref{ess:lemma:gen-dense}(b),
$M \times E
 \leq_e 
\widetilde M \times E$,
since $\widetilde M \times E$ is dense 
in $\widetilde M   \times   \widetilde E$.

(b) By Theorem~\ref{thm:min-ab}, $\widetilde M$ is compact, and so, by 
Lemma~\ref{ess:lemma:ewcomp-prod}, 
$\widetilde M \times E
 \leq_e 
\widetilde M \times G$. 
By Lemma~\ref{ess:lemma:gen-dense}(a), combining this with what has been 
shown in part (a) yields
$M \times E
 \leq_e 
\widetilde M \times G$.

(c) As $M \times G$ is dense in 
$\widetilde M \times G$,
by Lemma~\ref{ess:lemma:gen-dense}(b), part (b) implies that
$M \times E
 \leq_e 
M \times G$.
\end{proof}

\begin{theorem}  \label{ess:thm:perfmin}
Let $M$ be a minimal abelian group. The following 
statements are equivalent:

\begin{myromanlist}

\item
$M$ is perfectly minimal;

\item
for every topological group $G$, if
$E  \leq_e   G$ and $E$ is
minimal, then
$M \times E
 \leq_e 
M \times G$;

\item
for every abelian topological group $G$, if
$E  \leq_e   G$ and $E$ is
minimal, then
$M \times E
 \leq_e 
M \times G$;

\item
for every minimal abelian group $M^\prime$, the product
$M \times M^\prime$ is minimal.

\end{myromanlist}
\end{theorem}

\begin{proof}
The implication (i) $\Rightarrow$ (ii) follows by 
Lemma~\ref{ess:lemma:prodMin}(c), and the implication 
(ii) $\Rightarrow$ (iii) is trivial. 
The equivalence (i) $\Leftrightarrow$ (iv) is an immediate 
consequence of the following theorem of Stoyanov: 
{\slshape A topological group $G$ 
is perfectly minimal if and only if 
$G \times   (\mathbb{Z},\tau_p)$ is minimal 
for every prime $p$} (cf.~\cite{Stoy}).
We turn now to the remaining implication.

(iii) $\Rightarrow$ (iv): Let $M^\prime$ be a minimal abelian group. By 
Theorem~\ref{thm:min-ab}, $\widetilde M$ and 
$\widetilde{M^\prime}$ are compact, and 
$M  \leq_e   \widetilde{M}$ and
$M^\prime  \leq_e  
\widetilde{M^\prime}$. 
So, by (iii),
$M \times M^\prime
 \leq_e 
M \times \widetilde M^\prime $, 
and by Lemma~\ref{ess:lemma:ewcomp-prod},
$M \times \widetilde M^\prime
 \leq_e 
\widetilde M \times 
\widetilde M^\prime$. Therefore, by 
Lemma~\ref{ess:lemma:gen-dense}(a),
$M \times M^\prime
 \leq_e  
\widetilde M \times
\widetilde M^\prime $. Hence, by
Theorem~\ref{thm:min-ab}, $M \times 
M^\prime$ is minimal.
\end{proof}

\begin{remark} \label{ess:rem:prod}
It follows from Remark~\ref{ess:rem:ab-sub} and 
Theorem~\ref{ess:thm:perfmin} that every closed subgroup of a~perfectly 
minimal abelian group is perfectly minimal. 
\end{remark}

\section{Proof of Theorems~\ref{main:thm:realize} 
and~\ref{main:thm:const}}

\label{sect:proofs}

We prove Theorem~\ref{main:thm:realize} by establishing the following 
more elaborate statement.

\begin{Ltheorem*}[\ref{main:thm:realize}$'$] 
Let $A$ be an essential subgroup of a connected compact abelian group 
$C$. Then there exists an 
abelian group $G$ such that 
\begin{myalphlist}

\item
$G$ is pseudocompact;

\item
$A  \cong q(G)$;

\item 
$C   \cong    q(\widetilde G) = (\widetilde G)_0$,
and in  particular,
$\dim G  =  \dim C$;

\item
$G$ is minimal.

\end{myalphlist}

\vspace{6pt}

\noindent
Furthermore, if $A$ is perfectly minimal, then $G$ may be 
chosen to be perfectly minimal.
\end{Ltheorem*}

We first show how Theorem~\ref{main:thm:realize} follows from 
Theorem~\ref{main:thm:realize}$'$.

\begin{proof}[Proof of Theorem~\ref{main:thm:realize}.] Sufficiency of the  
conditions in (a) and (b) follows from Theorem~\ref{main:thm:realize}$'$, 
while sufficiency of (c) was already shown in 
\cite[Theorem~D$'$]{DikGL3}. Thus, we may turn to the necessity of the 
conditions.

(a) Let $G$ be a minimal pseudocompact abelian group. Then, by 
Theorem~\ref{thm:min-ab}, $G \leq_e  \widetilde  G$, and thus, by 
Lemma~\ref{ess:lemma:ab-cl}, 
$G   \cap   q(\widetilde G)  \leq_e  q(\widetilde G)$. 
By Theorem~\ref{prel:thm:connsum},
$q(G)= q(\widetilde G)  \cap  G$, and therefore
$q(G) \leq_e  q(\widetilde G)$.

(b) If $G$ is a perfectly minimal pseudocompact abelian group, then by 
Remark~\ref{ess:rem:prod}, $q(G)$ is also perfectly minimal, and by part 
(a), $q(G) \leq_e  q(\widetilde G)$.

(c) Let $G$ be a (perfectly) totally minimal pseudocompact abelian group. 
Then $G/q(G)$ is minimal, and by Theorem~\ref{prel:thm:qG}, $q(G)$ is 
dense in $q(\widetilde G)$. It follows from Remarks~\ref{ess:rem:ab-sub} 
and~\ref{ess:rem:prod} that $q(G)$ is (perfectly) totally minimal, being 
a closed subgroup of a (perfectly) totally minimal abelian group.
\end{proof}

We turn now to the proof of Theorem~\ref{main:thm:realize}$'$, and to that 
end, we recall two technical lemmas.

\begin{flemma}[{\cite[5.2]{DikGL3}}] \label{lemma:kernel}
For every infinite cardinal $\lambda$, there exists a 
pseudocompact zero-dimensional group $H$ such that:

\begin{myromanlist}

\item
$H$ is perfectly totally minimal;

\item
$r_0(\widetilde H / H)   \geq 
  2^\lambda$.
\end{myromanlist}

\end{flemma}

\begin{flemma}[{\cite[5.3]{DikGL3}}] \label{lemma:graph}
Let $K_1$ and $K_2$ be compact topological groups, and let
$h\colon K_1  \rightarrow  K_2$ be 
a~surjective homomorphism such that $\ker h$ is $G_\delta$-dense in $K_1$. 
Then the graph $\Gamma_h$ of $h$ is a $G_\delta\mbox{-}$dense 
subgroup of the product $K_1  \times  K_2$,
and in particular, $\Gamma_h$ is pseudocompact.
\end{flemma}

\begin{proof}[Proof of Theorem~\ref{main:thm:realize}$'$.]
Put $\lambda   =   w(C)$, and let $H$ be 
the 
group provided by Lemma~\ref{lemma:kernel}. Since
$r_0(\widetilde H / H)   \geq   
2^\lambda$,
the quotient $\widetilde H / H$ contains a free abelian group 
$F$ of rank $2^\lambda$. 
As $|C|  \leq  2^{\lambda}$, one may
pick a surjective homomorphism $h_1\colon F \rightarrow 
C$.
The group $C$ is divisible, because it is compact and connected 
(cf.~\cite[24.25]{HewRos}). Thus, $h_1$ can be extended to a surjective 
homomorphism
 $h_2\colon \widetilde H / H \rightarrow C$. 

Let $h\colon \widetilde H \rightarrow C$ denote 
the composition of $h_2$ with the canonical projection
$\widetilde H  \rightarrow \widetilde H / H$.
By Theorem~\ref{prel:thm:lcps}, $H$ is $G_\delta$-dense in 
$\widetilde H$, 
because $H$ is pseudocompact. Thus, $\ker h$ is $G_\delta$-dense in 
$\widetilde H$, because  
$H   \subseteq   \ker h$. Clearly, $h$ is 
surjective. Therefore, by Lemma~\ref{lemma:graph}, the graph
$\Gamma_h$ of $h$ is $G_\delta\mbox{-}$dense in the product
$\widetilde H \times  C$.

Put $G :=   \Gamma_h + (\{0\}  \times  A)$. Since $\Gamma_h$ is 
$G_\delta$-dense in $\widetilde H \times  C$ and  contained in $G$, the 
group $G$ is 
$G_\delta\mbox{-}$dense too.  Thus,  
$\widetilde G = \widetilde H \times  C$, and  by 
Theorem~\ref{prel:thm:lcps}, $G$ is pseudocompact. As $H$ is zero-dimensional, 
$q(\widetilde G)=(\widetilde G)_0   =   \{0\}   \times   C$, and by 
Theorem~\ref{prel:thm:connsum}(a), 
$q(G)  =   q(\widetilde G)  \cap  G  = \{0\}   \times   A$. 

We check now that $\dim G  = 
 \dim C$. Since $G$ is pseudocompact, by 
Theorem~\ref{prel:thm:lcps}, 
$\widetilde G   =   \beta 
G$,~and so 
$\dim G  =  \dim \beta G  
  =  \dim \widetilde G$ 
(cf.~\cite[7.1.17]{Engel6}).
As $H$ is zero-dimensional and pseudocompact, by 
Theorem~\ref{prel:thm:connsum}(b), 
$\dim\widetilde  H   =  0$.
Thus, by Yamanoshita's Theorem, 
$\dim \widetilde G  = 
\dim \widetilde H  +  \dim C 
 =   \dim C$
(cf.~\cite{Yamanoshita}, \cite[Corollary~2]{MostertDim},
and~\cite[3.3.12]{DikProSto}). Therefore,  
$\dim G  = \dim C$.

We turn now to minimality properties of $G$.
The group $G$ always  contains the product
$H \times  A$.
Since $C$ is compact, so is 
$\widetilde A   = 
\operatorname{cl}_C A$, and by Lemma~\ref{ess:lemma:ab-cl}(a),
$A\leq_e  \widetilde A$, 
because $A\leq_e  C$. Thus, by 
Theorem~\ref{thm:min-ab}, $A$ is minimal, and consequently,
$H \times  A$ is minimal, as $H$ is 
perfectly minimal. Therefore, by Lemma~\ref{ess:lemma:prodMin}(b),
$H \times  A
\leq_e 
\widetilde H \times  C 
  =   \widetilde G$. Hence,
by Lemma~\ref{ess:lemma:ab-cl}(a),
$G \leq_e  \widetilde G$,
and by Theorem~\ref{thm:min-ab}, $G$~is minimal.

Suppose now that $A$ is perfectly minimal, and let 
$M^\prime$ be a minimal 
abelian group. Then 
$H \times  A$ is perfectly minimal, 
and so 
$H \times  A   
\times M^\prime$ is minimal. Thus, by
Lemma~\ref{ess:lemma:prodMin}(b),
$H \times A  
\times  M^\prime  \leq_e 
\widetilde G   \times   
\widetilde{M^\prime}$.
Therefore, by Lemma~\ref{ess:lemma:ab-cl}(a),
$G  
\times  M^\prime  \leq_e \widetilde G   \times      
\widetilde{M^\prime}$, because
$G  \times  M^\prime$ contains
$H \times  A  
\times  M^\prime$. Hence, by 
Theorem~\ref{thm:min-ab}, $G  \times  
M^\prime$ is minimal. 
Consequently, by Theorem~\ref{ess:thm:perfmin},
$G$ is perfectly minimal, as desired. 
\end{proof}

In preparation for the proof of Theorem~\ref{main:thm:const}, 
we recall some 
remarkable properties of the Pontryagin dual of the discrete group of 
the rationals.
Let $\mathbb{T}  :=  \mathbb{R}/\mathbb{Z}$, 
and put $\widehat{\mathbb{Q}}   := 
\hom(\mathbb{Q},\mathbb{T})$, 
equipped with the topology of pointwise convergence. The group
$\widehat{\mathbb{Q}}$ is compact, and  it is connected, because 
$\mathbb{Q}$ is torsion-free (cf.~\cite[3.3.8]{DikProSto}). 

\begin{flemma} {\rm (\cite[3.6.2, 3.6.5]{DikProSto})}
\label{lemma:Zann-ess}
Let $\mathbb{Z}^\perp  := 
\{\chi  \in   \widehat{\mathbb{Q}} \mid
\chi(\mathbb{Z}) = 0\}$ denote the annihilator
 of $\mathbb{Z}$ in $\widehat{\mathbb{Q}}$.

\begin{myalphlist}

\item
$\mathbb{Z}^\perp$ is a compact essential subgroup 
of $\widehat{\mathbb{Q}}$;

\item
$\mathbb{Z}^\perp   \cong   
\prod\limits_{p\in \mathbb{P}} \mathbb{Z}_p$, and in 
particular, 
$\dim \mathbb{Z}^\perp  =   0$;

\item
$\widehat{\mathbb{Q}}/\mathbb{Z}^\perp   \cong 
 \mathbb{T}$, and in particular,
$\dim \widehat{\mathbb{Q}}  =   1$.

\end{myalphlist}
\end{flemma}

\begin{proof}[Proof of Theorem~\ref{main:thm:const}.]
Suppose that $n  < \omega$. Put
$C_n :=   (\widehat{\mathbb{Q}})^n$ and 
$A_n :=  (\mathbb{Z}^\perp)^n $.
By Lemma~\ref{lemma:Zann-ess}, $\mathbb{Z}^\perp$ is a compact essential 
subgroup of $\widehat{\mathbb{Q}}$. Thus, by 
applying Lemmas~\ref{ess:lemma:ewcomp-prod} 
and~\ref{ess:lemma:gen-dense}(a) repeatedly, one obtains that 
$A_n$ is a compact essential subgroup of $C_n$. 
It follows from  Lemma~\ref{lemma:Zann-ess} 
and Yamanoshita's Theorem that 
$\dim C_n  =  n$ (cf.~\cite{Yamanoshita}, \cite[Corollary~2]{MostertDim},
and~\cite[3.3.12]{DikProSto}). By 
Theorem~\ref{main:thm:realize}$'$, there exists a perfectly minimal 
pseudocompact 
group $G_n$ such that $A_n   \cong   q(G_n)$,
$C_n   \cong   q(\widetilde{G}_n)$, and 
$\dim G_n  =  \dim C_n  =  n$.  
Clearly, $q(G_n)$ is not only not dense, but in fact
it is closed in $q(\widetilde{G}_n)$.

For $n =\omega$, put
$G_\omega  =   G_1 
  \times   \mathbb{T}^\omega $. Then
$\dim G_\omega   \geq  \omega$, and
$q(G_\omega)  =   
q(G_1)  \times   \mathbb{T}^\omega \cong 
  \mathbb{Z}^\perp  \times  \mathbb{T}^\omega $, while
\begin{align*}
q(\widetilde G_\omega)=
(\widetilde G_\omega)_0 = (\widetilde G_1 \times \mathbb{T}^\omega)_0 = 
(\widetilde G_1)_0 \times \mathbb{T}^\omega \cong 
\widehat{\mathbb{Q}} \times \mathbb{T}^\omega.
\end{align*}
Thus, $q(G_\omega)$ is not dense in $q(\widetilde G_\omega)$, as 
required. The last statement follows from Theorem~\ref{prel:thm:qG}.
\end{proof}

\section*{Acknowledgements}

We are grateful to Karen Kipper for her kind help in proofreading this 
paper for grammar and punctuation. We wish to thank the anonymous referee 
for the constructive comments that led to an improved presentation of this 
paper.

{\footnotesize

\bibliography{notes,notes2,notes3}
}

\begin{samepage}

\bigskip
\noindent
\begin{tabular}{l @{\hspace{1.8cm}} l}
Department of Mathematics and Computer Science & Halifax, Nova Scotia\\
University of Udine & Canada\\
Via delle Scienze, 208 -- Loc. Rizzi, 33100 Udine &  \\
Italy &  \\ 
& \\
\em email: dikranja@dimi.uniud.it  &
\em email: lukacs@topgroups.ca
\end{tabular}

\end{samepage}

\end{document}